\documentclass[12pt, reqno]{amsart}
\usepackage{amssymb, amsmath, amsfonts, amscd, calligra, mathrsfs}
\usepackage[raiselinks=false,colorlinks=true,citecolor=blue,urlcolor=blue,
linkcolor=blue,bookmarksopen=true,pdftex]{hyperref}
\usepackage[dvipsnames]{xcolor}
\usepackage{tikz}\usepackage{tikz-cd}
\usepackage{enumerate}
\usepackage[mathscr]{eucal}
\newtheorem{theorem}{Theorem}[section]
\newtheorem{proposition}[theorem]{Proposition}
\newtheorem{lemma}[theorem]{Lemma}
\newtheorem{remark}[theorem]{Remark}

\newtheorem{corollary}[theorem]{Corollary}

\newcommand{\Sec}{\textrm{Sec}}
\newcommand{\closed}{\textrm{closed}}
\newcommand{\subvariety}{\textrm{subvariety}}
\newcommand{\rational}{\textrm{rational}}
\newcommand{\equivalence}{\textrm{equivalence}}

\newcommand{\stable}{\textrm{stable}}
\begin{document}
\title[Brill-Noether loci for moduli space]{A note on the Brill-Noether loci of small codimension in moduli space of stable bundles}
\author[P. Biswas]{Pritthijit Biswas} 
\address{The Institute of Mathematical Sciences (HBNI), C.I.T. Campus, Taramani, Chennai, 600113. Tamil Nadu, India.}
\email{pbiswas@imsc.res.in\\
pritthibis06@gmail.com}
\author[J.N. Iyer]{Jaya NN Iyer}
\address{The Institute of Mathematical Sciences (HBNI), C.I.T. Campus, Taramani, Chennai, 600113. Tamil Nadu, India.}
\email{jniyer@imsc.res.in}
\subjclass[1973]{14D20}
\keywords{Brill-Noether Loci for Moduli of vector bundles; Chow groups} 
\begin{abstract}
    Let $X$ be a smooth projective curve of genus $g$ over the field $\mathbb{C}$. Let $M_{X}(2,L)$ denote the moduli space of stable rank $2$ vector bundles on $X$ with fixed determinant $L$ of degree $2g-1$. Consider the Brill-Noether subvariety $W^{1}_{X}(2,L)$ of $M_{X}(2,L)$ which parametrises stable vector bundles having at least two linearly independent global sections. In this article, for generic $X$ and $L$, we show that $W^{1}_{X}(2,L)$ is stably-rational when $g=3$, unirational when $g=4$, and rationally chain connected by Hecke curves, when $g\geq 5$. We also show triviality of low dimensional rational Chow groups of an associated Brill Noether hypersurface.
\end{abstract}
\maketitle
\section{introduction}
Let $X$ be a smooth projective curve over $\mathbb{C}$ with genus $g\geq 2$. Let $L$ be a line bundle on $X$ of degree $d$. Let $n\geq 1$ be an integer. Let $M_{X}(n,L)$ denote the moduli space of stable bundles on $X$ of rank $n$ and fixed determinant $L$ and $M_{X}(n,d)$ denote the moduli space of stable bundles on $X$ of rank $n$ and degree $d$. 
It is known that $M_{X}(n,L)$ and $M_{X}(n,d)$ are smooth projective varieties of dimensions $(n^{2}-1)(g-1)$ and $n^{2}(g-1)+1$ respectively when $(n,d)=1$. 

There has been a considerable amount of study about non-emptiness, irreducibility and dimension of the {\it Brill-Noether subvarieties} $W^{r}_{X}(n,d)$ and $W^{r}_{X}(n,L)$ of $M_{X}(n,d)$ and $M_{X}(n,L)$ respectively (see \S 2 for notations, \cite{bgn},\cite{n1}, \cite{sundaram}, \cite{tb1} and \cite{tb2}). On the other hand, rationality of the stable moduli space with fixed determinant over a curve has also been well studied, the most well known results are due to Newstead, Schofield and King, see \cite{n2},\cite{n3} and \cite{sk}.

In this article, for small values of $g$, we study unirationality and stably-rationality of the variety $W^{1}_{X}(2,L)$ when $X$ and $L$ are both generic and $\deg(L)=d=2g-1$. 

 More precisely we obtain the following theorem:
\begin{theorem}
    Let $X$ be a smooth projective curve of genus $g\geq 2$ and $L$ be a fixed line bundle on $X$ of degree $2g-1$. Assume that $X$ and $L$ are both generic. Then the following are true:

    $(i)$ If $g=3$, then $W^{1}_{X}(2,L)$ is stably-rational.    
    
    $(ii)$ If $g=4$, then $W^{1}_{X}(2,L)$ is unirational.  
    
    $(iii)$ If $g\geq 5$, then $W^1_X(2,L)$ is rationally chain connected. It is connected by chain of Hecke curves.
\end{theorem}

Assume $\deg(L)=d=2g-1$. Let $\mathbb{P}_{L}$ denote the extension space $\mathbb{P}(Ext^{1}(L,\mathcal{O})^{*})$ and $\mathcal{H}$ denote the closed subscheme of $\mathbb{P}_{L}$ parameterising all extensions $0\rightarrow \mathcal{O}\rightarrow E\rightarrow L\rightarrow 0$ with $h^{0}(E)\geq 2$. Then $\mathcal{H}$ is a degree $g$ hypersurface of $\mathbb{P}_{L}$ \cite[Definition-Claim 4.6]{b}.

Then we prove the following.
\begin{theorem}
Let $X$ be a smooth projective curve of genus $g\geq 2$ and $L$ be a line bundle on $X$ of degree $2g-1$. Then the following holds:

If $g\geq 2$, then $CH_{0}(\mathcal{H})_{\mathbb{Q}}\cong \mathbb{Q}$ and if $g\in\{3,4,5,6\}$, then $CH_{1}(\mathcal{H})_{\mathbb{Q}}\cong \mathbb{Q}$.

\end{theorem}

The proofs are based on analysing Bertram's work \cite{b} on the space of extensions $\mathbb{P}_{L}$ and the associated rational map $\phi_L:\mathbb{P}_{L} \dashrightarrow M_X(2,L)$.  The singular hypersurface $\mathcal{H}$ dominates $W^1_X(2,L)$ and is in fact generically a $\mathbb{P}^1$-fibration.  The study of $W^1_X(2,L)$ is reduced to analysing the hypersurface $\mathcal{H}$.

\section{Notations and Conventions}
In this paper we will use the following.

$(i)$ All schemes which are considered in this paper are over the field $\mathbb{C}$ of complex numbers.

$(ii)$ $X=$ Smooth projective irreducible curve of genus $g$ over $\mathbb{C}$.

$(iii)$  $L=$ Line bundle on $X$.

$(iv)$ $M_{X}(n,d)=$ Moduli space of stable vector bundles of rank $n$ and degree $d$ over $X$ when $n\geq 1$.


$(v)$  $M_{X}(n,L)=$ Moduli space of stable vector bundles of rank $n$ and determinant $L$ over $X$ when $n\geq 1$.


$(vi)$  $W^{r}_{X}(n,d):=\{E: E$ is stable of rank $n$ degree $d$ with $h^{0}(E)\geq r+1\}\subseteq M_{X}(n,d)$; 

$W^{r}_{X}(n,L):=\{E: E$ is stable of rank $n$ determinant $L$ with $h^{0}(E)\geq r+1\}\subseteq M_{X}(n,L)$ when $n\geq 1$. These are the well known Brill-Noether varieties, see \cite[\S1, 1.1]{bgn}. 

$(vii)$  The number $\rho^{r}_{n,d}$ will denote the expected dimension of $W^{r}_{X}(n,d)$ $(n\geq 1)$. It is given by
\[\rho^{r}_{n,d}= n^{2}(g-1)+1-(r+1)(r+1-d+n(g-1)).\]
The number $\rho^{r}_{n,L}$ will denote the expected dimension of $W^{r}_{X}(n,L)$ $(n\geq 1)$. It is given by
\[\rho^{r}_{n,L}:= (n^{2}-1)(g-1)-(r+1)(r+1-d+n(g-1)).\]

$(viii)$ A variety $V$ is {\it unirational} if there is a dominant rational map 
\[\phi: \mathbb{P}^{N}\dashrightarrow V\] for some non negative integer $N$. Moreover $V$ is {\it rational} if such a $\phi$ exists and is birational.

$(ix)$ A variety $V$ is {\it stably-rational} if $V\times \mathbb{P}^{m}$ is rational for some $m\in \mathbb{Z}_{\geq 0}$.

$(x)$ A variety $V$ is called {\it rationally connected} if any two general points of $V$ can be joined by an irreducible rational curve and more generally $V$ is called {\it rationally chain connected} if any two general points of $V$ can be joined by a finite chain of connected rational curves (see \cite{Kollar}).

$(xi)$  For a variety $V$ of dimension $n$ and any integer $k$ with $0\leq k\leq n$, the $k$-th {\it Chow group of $V$} (see \cite{f}) is defined by:
\[CH_{k}(V):=\mathbb{Z}\{W\subseteq V~ \closed ~\subvariety : \dim W=k\}/ \rational ~\equivalence.\]

\vspace{0.5em}
$(xii)$ For a variety $V$ of dimension $n$ and for any integer $k$ with $0\leq k\leq n$, we define: 

$CH^{k}(V):=CH_{n-k}(V)$.

$CH_{k}(V)_{\mathbb{Q}}:=CH_{k}(V)\otimes_{\mathbb{Z}} \mathbb{Q}$.

$CH^{k}(V)_{\mathbb{Q}}:=CH_{n-k}(V)_{\mathbb{Q}}$.

$CH^{k}(V)_{hom}:=\{W\in CH^{k}(X): W$ is homologous to zero\}.

\section{Rank 2 stable vector bundles and the extension map}

Let $X$ be a smooth projective connected curve of genus $g\geq 2$ over $\mathbb{C}$ and $L$ will be a fixed line bundle on $X$ of degree $2g-1$. We  denote the projective space $\mathbb{P}(Ext^{1}(L,\mathcal{O})^{*})$ by $\mathbb{P}_{L}$. 

Let $\pi_{X}$ and $\pi_{L}$ denote the projections from $X\times \mathbb{P}_{L}$ to $X$ and $\mathbb{P}_{L}$ respectively. Then there is a natural Poincar\'e extension on $X\times \mathbb{P}_{L}$, given by:
\begin{equation}\label{a}
    0\rightarrow \pi_{L}^{*}\mathcal{O}_{\mathbb{P}_{L}}(1)\rightarrow \mathcal{E}_{L}\rightarrow \pi_{X}^{*}L\rightarrow 0
    \end{equation}
    which satisfies the property that for any closed point $x\in \mathbb{P}_{L}$, the restriction of $(\ref{a})$ to $X\cong X\times x$ corresponds to $x$. In fact, one checks that under the isomorphism $H^{1}(X\times \mathbb{P}_{L}, \pi_{X}^{*}L^{*}\otimes \pi_{L}^{*}\mathcal{O}_{\mathbb{P}_{L}}(1))\cong H^{1}(X,L^{*})\otimes H^{1}(X,L^{*})^{*}$ given by the K\"unneth formula, the identity map from $H^{1}(X,L^{*})\rightarrow H^{1}(X,L^{*})$ corresponds to an extension with the above property. See \cite[Definition-claim 3.1]{b}. Thus there is a rational map $\phi_{L}:\mathbb{P}_{L}\dashrightarrow M_{X}(2,L)$.

    If $D$ is an effective divisor on $X$, then there is a natural inclusion $\mathcal{O}(-D)\hookrightarrow \mathcal{O}$. This induces a natural inclusion $L(-D)\hookrightarrow L$ which on dualising gives a natural map $h:L^{*}\rightarrow L^{*}(D)$. 
    We denote the {\it linear span} of an effective divisor $D$ by $\overline{D}$ and define it by (see \cite{ln})    \[\overline{D}:=\mathbb{P}(Ker(h_{*}:H^{1}(X,L^{*})\rightarrow H^{1}(X,L^{*}(D)))^{*})\subseteq \mathbb{P}_{L}.\]
    Note that $h_{*}: H^{1}(X,L^{*})\rightarrow H^{1}(X,L^{*}(D))$ is a surjection and hence $\dim(\overline{D})= h^{1}(L^{*})-h^{1}(L^{*}(D))-1$.

As $g\geq 2$ and $\deg(L)=2g-1$, $\deg(K\otimes L)\geq 2g+1$. Thus there is a closed immersion $\phi:X\hookrightarrow \mathbb{P}(H^{0}(X,K\otimes L)=Ext^{1}(L,\mathcal{O})^{*})$. 
    For any $r\geq 1$, $\Sec^{r}(X)$ will denote the {\it secant variety} of $(r-1)$-planes through $r$ points of $X$ inside $\mathbb{P}_{L}$, and is defined by 
\[\Sec^{r}(X):=\bigcup_{D\geq 0,~ degD=r}\overline{D}.\]

Note that $\Sec^{1}(X)=X$. Moreover for any $r\geq 1$, $\Sec^{r}(X)$ is at most of dimension $2r-1$ and there is an inclusion $\Sec^{r}(X)\subseteq \Sec^{r+1}(X)$.

\begin{lemma}\label{3.1}
The closed irreducible subvariety $\Sec^{g-1}(X)$ of $\mathbb{P}_{L}$ is precisely the locus of all extensions which determine rank $2$ bundles that are not stable.
    \end{lemma}
    \begin{proof}
        See \cite[Pg. 451]{b}.
    \end{proof}
 
 Therefore we see that the domain of $\phi_{L}$ is $\mathbb{P}_{L}\backslash \Sec^{g-1}(X)$. We will denote the {\it indiscrepancy locus} of $\phi_{L}$ by $\mathcal{S}$ which is defined by :
\begin{equation}\label{bb}
\mathcal{S}=\Sec^{g-1}(X).
\end{equation}
Moreover $\phi_{L}$ can be resolved into a birational morphism from a suitable variety $\widetilde{\mathbb{P}}_{L}$ onto $M_{X}(2,L)$ by a sequence of blow-ups commencing from blowing up the curve $X$ in $\mathbb{P}_{L}$. More precisely we have the following theorem of Bertram.

\begin{theorem}\label{5}\cite[Theorem 1]{b}
Let $X$ be a smooth projective curve of genus $g\geq 2$ and $L$ be a line bundle on $X$ of degree $2g-1$. Then the following are true:

$(i)$ There are  blow-ups $BL_{r}(\mathbb{P}_{L})\rightarrow BL_{r-1}(\mathbb{P}_{L})$ along smooth centres  with exceptional divisor $E_{r}$ for all $1\leq r\leq g-1$, with $BL_{0}(\mathbb{P}_{L})=\mathbb{P}_{L}$ and $BL_{1}(\mathbb{P}_{L})\rightarrow \mathbb{P}_{L}$ is the blow-up of $\mathbb{P}_{L}$ along $X$.

$(ii)$ There is a morphism $\Phi_{L}:\widetilde{\mathbb{P}}_{L}\rightarrow M_{X}(2,L)$ such that as rational maps $\Phi_{L}=\phi_{L}\circ \sigma$, where $\sigma:\widetilde{\mathbb{P}}_{L}\rightarrow \mathbb{P}_{L}$ is the successive composition of the blow-up maps in $(i)$. Moreover if for every $r$ with $1\leq r\leq g-1$, $\widetilde{E}_{r}$ denotes the inverse image of $E_{r}$ under the composition of the successive blow-ups above $BL_{r}(\mathbb{P}_{L})$ in $\widetilde{\mathbb{P}}_{L}$, then $\sigma(\widetilde{E}_{r})\subseteq \Sec^{r}(X)$, $\widetilde{E}_{1}\cup...\cup\widetilde{E}_{g-2}\cup {E}_{g-1}=\sigma^{-1}(\Sec^{g-1}(X))$ and $\sigma_{|U_{L}}:U_{L}\rightarrow \mathbb{P}_{L}\backslash \Sec^{g-1}(X)$ is an isomorphism, where $U_{L}=\widetilde{\mathbb{P}}_{L}\backslash \widetilde{E}_{1}\cup...\cup\widetilde{E}_{g-2}\cup {E}_{g-1}$.

$(iii)$ If $D$ is an effective divisor and $x\in\overline{D}$ is not in the span of any proper subdivisor of $D$, then there is a natural isomorphism $\sigma^{-1}(x)\cong \widetilde{\mathbb{P}}_{L(-2D)}$ and when restricted to $\sigma^{-1}(x)$, $\Phi_{L}$ coincides with the map 
\[\Phi_{L(-2D)}\otimes \mathcal{O}(D): \widetilde{\mathbb{P}}_{L(-2D)}\rightarrow M_{X}(2,L(-2D))\rightarrow M_{X}(2,L).\]

$(iv)$ There is a universal bundle $\widetilde{\mathscr{E}}_{L}$ on $\widetilde{\mathbb{P}}_{L}\times X$ which realises the map $\Phi_{L}$. 
\end{theorem}
Furthermore $\Phi_{L}$ is surjective, all the fibres of $\Phi_{L}$ are connected and there is an isomorphism ${\Phi_{L}}_{*}\mathcal{O}_{\widetilde{\mathbb{P}}_{L}}\cong \mathcal{O}_{M_{X}(2,L)}$ \cite[ Proposition $(4.5)$]{b}. Therefore as $\widetilde{\mathbb{P}}_{L}$ and $M_{X}(2,L)$ are of the same dimension $3g-3$, $\Phi_{L}$ is birational. Thus by Theorem $(\ref{5})(iii)$, $\phi_{L}:\mathbb{P}_{L}\backslash \Sec^{g-1}(X)\rightarrow M_{X}(2,L)$ is also birational yielding the rationality of $M_{X}(2,L)$. 

On pushing down the Poincar\'e extension $(1)$ to $\mathbb{P}_{L}$, we get the following exact sequence of sheaves on $\mathbb{P}_{L}$.
    \begin{equation}\label{b}
          0\rightarrow \mathcal{O}_{\mathbb{P}_{L}}(1)\rightarrow {\pi_{L}}_{*}\mathcal{E}_{L}\rightarrow H^{0}(X,L)\otimes \mathcal{O}_{\mathbb{P}_{L}}\overset{\lambda}\longrightarrow H^{1}(X,\mathcal{O})\otimes \mathcal{O}_{\mathbb{P}_{L}}(1)
      \end{equation}
      \[\rightarrow R^{1}{\pi_{L}}_{*}\mathcal{E}_{L}\rightarrow H^{1}(X,L)\otimes \mathcal{O}_{\mathbb{P}_{L}}\rightarrow 0\]
      
      As $\deg(L)=2g-1$, the map $\lambda$ in $(\ref{b})$ is a $g\times g$ matrix of linear forms on $\mathbb{P}_{L}$. Let $\mathcal{H}$ be the closed subscheme of $\mathbb{P}
_{L}$ determined by the ideal sheaf locally generated by $\det(\lambda)$. Therefore $\mathcal{H}=\{[[0\rightarrow \mathcal{O}\rightarrow E\rightarrow L\rightarrow 0]]\in \mathbb{P}_{L}| h^{0}(E)\geq 2\} \subset \mathbb{P}_L$. 

\begin{lemma}\label{3.3}
$\mathcal{H}$ is not a cone and is a singular hypersurface of degree $g$ in $\mathbb{P}_{L}$.
\end{lemma}
\begin{proof}
    The map $\lambda$ in $(\ref{b})$ is non degenerate and hence $\mathcal{H}$ is not a cone. Other assertions follow from \cite[Pg. 464]{b}.
\end{proof}

Our goal now is to relate the Brill-Noether Locus $W^{1}_{X}(2,L)$, with the hypersurface $\mathcal{H}$ via the birational map $\phi_{L}$ and then deduce information about it from that of $\mathcal{H}$.

Recall that $\mathcal{S}$ is the indiscrepancy locus of $\phi_{L}$ as defined in $(\ref{bb})$. We have the following proposition.
\begin{proposition}\label{35}
    Assume that $X$ and $L$ are generic. Then the following are true:
    
    $(i)$ $\mathcal{H}\nsubseteq \mathcal{S}$.

    $(ii)$ $W^{1}_{X}(2,L)$ is irreducible of codimension $2$ in $M_{X}(2,L)$ and ${\phi_{L}}_{|\mathcal{H}}:\mathcal{H}\dashrightarrow W^{1}_{X}(2,L)$ is dominant.

    $(iii)$ When $g\geq 3$, there is an open subset $W^{0}\subseteq W^{1}_{X}(2,L)$ such that $\mathcal{H}^{0}:=\phi_{L}^{-1}(W^{0})\rightarrow W^{0}$ is a $\mathbb{P}^{1}$-fibration.
\end{proposition}
\begin{proof}
$(i)$: If $\mathcal{H}\subseteq \mathcal{S}$, then $\dim(\mathcal{H})\leq \dim(\mathcal{S})$. As $\mathcal{S}=\Sec^{g-1}(X)$, $\dim(\mathcal{S})\leq 2g-3$ but $\dim(\mathcal{H})=3g-4$. This is a contradiction as $g\geq 2$.

$(ii)$: Recall from \cite[Theorem]{tb1} that if $X$ is generic, then $W^{1}_{X}(2,2g-1)$ is an irreducible variety of the expected dimension $4g-5$. Furthermore from \cite[Theorem $(1.1)$]{tb2} if $L$ is generic and $\deg(L)=2g-1$ then we deduce that $W^{1}_{X}(2,L)$ is non empty, irreducible and has the expected dimension $3g-5$. In other words, the codimension of $W^{1}_{X}(2,L)$ in $M_{X}(2,L)$ is $2$. 

The proof of  \cite[Theorem]{tb1} shows that a generic $E$ in $W^{1}_{X}(2,L)$ satisfies $h^{0}(E)=2$ and in fact the  hypersurface $\mathcal{H}$ contains an extension class $[0\rightarrow \mathcal{O}\rightarrow E\rightarrow L\rightarrow 0]$ with $E$ being a stable rank $2$ vector bundle with $h^{0}(E)=2$.   This implies that ${\phi_{L}}_{|\mathcal{H}}:\mathcal{H}\dashrightarrow W^{1}_{X}(2,L)$ is dominant.

 $(iii)$: Using \cite[Lemma $(4.1)$]{b}, when $E$ is a rank $2$ stable vector bundle on $X$ with $h^{0}(E)=2$, there is a rational injective map 
\[e:\mathbb{P}(H^{0}(X,E)^{*})\dashrightarrow \mathbb{P}_{L}\]
given by $s\mapsto [0\rightarrow \mathcal{O}\stackrel{s}{\rightarrow} E\rightarrow L\rightarrow 0]$. Since $\mathbb{P}(H^{0}(X,E)^{*})=\mathbb{P}^{1}$, this map $e$ is in fact a morphism. 

Furthermore, $e(\mathbb{P}(H^{0}(X,E)^{*}))$ does not intersect $\mathcal{S}$ for a generic $E$ in $W^{1}_{X}(2,L)$. Otherwise if $\mathcal{S}$ intersects each $e(\mathbb{P}(H^{0}(X,E)^{*}))$, then $\dim(\mathcal{S}\cap \mathcal{H}) \geq \dim(W^{1}_{X}(2,L))=3g-5$. However $\dim(\mathcal{S})\leq 2g-3$. This contradicts $g\geq 3$. Hence the fibre $\phi_{L}^{-1}(E)$ is $\mathbb{P}^{1}$ for generic $E$.

This gives a $\mathbb{P}^{1}$-fibration:\[{\phi_{L}}_{|\mathcal{H}^{0}}:\mathcal{H}^{0}\rightarrow W^{0}\]
where $W^{0}=\{E\in W^{1}_{X}(2,L):h^{0}(E)=2, E ~\stable \}\subseteq W^{1}_{X}(2,L)$ and $\mathcal{H}^{0}:=\phi_{L}^{-1}(W_{0})\subseteq \mathcal{H}$ are open sets. 
\end{proof}

Consider the restriction $\mathcal{E}^{0}$ of the Poincar\'e bundle $\mathcal{E}$ on $X\times M_{X}(2,L)$ to $X\times W^{0}$. Let $p_{2}:X\times W^{0}\rightarrow W^{0}$ be the second projection. 
\begin{corollary}\label{z} 
Assume that $X$ and $L$ are generic. Then the following are true :

$(a)$ If $g\geq 2$, then unirationality of $\mathcal{H}$ implies unirationality of $W^{1}_{X}(2,L)$. 

$(b)$ If $g\geq 3$, then $\mathcal{H}$ is birational to the projective bundle 
    \[\mathbb{P}({p_{2}}_{*}\mathcal{E}^{0})\rightarrow W^{0}\]
\end{corollary}
\begin{proof}
    $(a)$ and $(b)$ follow from Proposition $(\ref{35})$ parts $(ii)$ and $(iii)$ respectively.
    \end{proof}

Recall that a variety $V$ is called {\it stably-rational} if $V\times \mathbb{P}^{m}$ is rational for some non negative integer $m$.
 \begin{proposition}\label{notcone}
       Let $X$ be a smooth projective curve of genus $3$ and $L$ be a line bundle on $X$ of degree $5$. Assume that $X$ and $L$ are generic. Then $W^{1}_{X}(2,L)$ is stably-rational.
    \end{proposition}
    \begin{proof}
        In this case, by Lemma $(\ref{3.3})$, $\mathcal{H}$ is a singular cubic hypersurface which is not a cone and it is well known that $\mathcal{H}$ is rational in this case.  In fact if $p$ is a singular point on $\mathcal{H}$, then the variety $V_{\mathcal{H}}$ of lines through $p$ is a projective space of dimension $dim(\mathbb{P}_L)-1 = dim(\mathcal{H})$. A general line $l\in V_{\mathcal{H}}$ intersects $\mathcal{H}$ at exactly one point. This corresponds to a birational map $V_{\mathcal{H}} \rightarrow \mathcal{H}$. This proves rationality of $\mathcal{H}$. By Corollary $(\ref{z})(b)$, $\mathcal{H}$ is birational to the projective bundle  
        \[\mathbb{P}({p_{2}}_{*}\mathcal{E}^{0})\rightarrow W^{0}.\]
This is a Zariski locally trivial fibration. Hence rationality of $\mathcal{H}$ implies stably-rationality of $W^{1}_{X}(2,L)$.       
    \end{proof}
    
    \subsection{Genus four case}
    
 Our next goal is to show that the hypersurface $\mathcal{H}$ is unirational when $g=4$.  We start by recalling a few well known facts about Grassmannians of projective linear subspaces of a projective space and incidence correspondences. See \cite[\S 2 and Theorem 3.3]{AK}, for details and proofs.

Let $r, d, n$ be integers with $d,n\geq 1$ and $1\leq r\leq n-1$.  Let $V:=H^{0}(\mathbb{P}^{n},\mathcal{O}(1))$. Clearly $\mathbb{P}^{n}\cong \mathbb{P}(V)$ and let $\mathbb{P}$ denote the scheme $\mathbb{P}(S^{d}(V^{*}))$. Then $\mathbb{P}$ parametrises all the degree $d$ hypersurfaces of $\mathbb{P}^{n}$. 

Let $\mathbb{G}(r,n)$ denote the Grassmann variety parametrising the $r$-dimensional projective linear subspaces of $\mathbb{P}^{n}$. Consider the incidence variety
 \[I(r,n,d):=\{(X,\Lambda)\in \mathbb{P}\times \mathbb{G}(r,n) : \Lambda\subseteq X\}.\] 

Let $p_{1}$ and $p_{2}$ denote the natural projections of $I(r,n,d)$ to $\mathbb{P}$ and $\mathbb{G}(r,n)$ respectively. The inclusion $I(r,n,d)\hookrightarrow \mathbb{P}\times \mathbb{G}(r,n)$ makes $I(r,n,d)\overset{p_{2}}\longrightarrow \mathbb{G}(r,n)$ into a projective subbundle of the trivial projective bundle $\mathbb{P}\times \mathbb{G}(r,n)\rightarrow \mathbb{G}(r,n)$ over $\mathbb{G}(r,n)$ of codimension $\binom {r+d}{d}$. 

In particular $I(r,n,d)$ is non empty and has dimension $(r+1)(n-r)+\binom{n+d}{d}-\binom{r+d}{d}-1$.  However if $d\geq 2$, then $p_{1}$ is surjective if and only if the following two conditions hold:

$(i)$ If $d\geq 3$, then $(r+1)(n-r)\geq \binom{r+d}{d}$

$(ii)$ If $d=2$, then $n\geq 2r+1$.

On the other hand, a conjecture of A. Conte states that, any degree $d$ projective hypersurface which contains a projective linear subspace of dimension $d-2$ is unirational. The conjecture is true for $d=4,5$, see \cite[pg 312]{M}. We obtain the following proposition.

\begin{proposition}\label{12}
Let $X$ be a smooth projective curve of genus $g\geq 2$ and $L$ be a line bundle on $X$ of degree $2g-1$. When the genus $g=4$, the quartic hypersurface $\mathcal{H}$ is unirational.
\end{proposition}
\begin{proof}
    Note that for fixed natural numbers $r,n,d$ with $1\leq r\leq n-1$ surjectivity of $p_{1}:I(r,n,d)\rightarrow \mathbb{P}$ means that any degree $d$ hypersurface in $\mathbb{P}^{n}$ contains a $r$- dimensional projective linear subspace of $\mathbb{P}^{n}$. Recall that $\mathcal{H}$ is a degree $g$ hypersurface in $\mathbb{P}^{3g-3}$. On choosing $(r,n,d)$ to be $(g-2,3g-3,g)$, and putting $g=4$ we see that $(r+1)(n-r)-\binom{r+d}{d}= 3(12-5)-\binom{6}{4}=21-15\geq 0$. This shows that any degree $g$ hypersurface in $\mathbb{P}^{3g-3}$ contains a $\mathbb{P}^{g-2}$, when $g=4$. Now applying \cite[Pg.312]{M} to the case $d=g=4$, we get that the hypersurface $\mathcal{H}$ is unirational.
\end{proof}

\begin{remark}However when $g=5$ similar computation as above shows that the map $p_{1}:I(r,n,d)\rightarrow \mathbb{P}$ is not surjective where $(r,n,d)=(g-2,3g-3,g)=(3,12,5)$. Therefore we do not know about existence of a $\mathbb{P}^{3}\subset \mathcal{H} \subset \mathbb{P}^{12}$. However in the next section we will discuss rational chain connectedness of $\mathcal{H}$, when $g\geq 5$ 
\end{remark}

We summarise above results as follows.

\begin{theorem}\label{main}
    Let $X$ be a smooth projective curve of genus $g\geq 2$ and $L$ be a fixed line bundle on $X$ of degree $2g-1$. Assume that $X$ and $L$ are generic. Then the following are true:

    $(i)$ If $g=3$, then $W^{1}_{X}(2,L)$ is stably-rational.

    $(ii)$ If $g=4$, then $W^{1}_{X}(2,L)$ is unirational. 
\end{theorem}
\begin{proof}
   $(i)$ is Proposition $(\ref{notcone})$ and $(ii)$ follows from Proposition $(\ref{12})$ and Corollary $(\ref{z})(a)$. 
\end{proof} 

\section{Rational chain connectedness, when  $g\geq 5$}

Assume $g\geq 2$.
 Recall that a variety $V$ is rationally connected if given any two points $x,y \in V$, there is an irreducible  rational curve $C$ on $V$ joining $x$ and $y$. Rationally chain connected will mean that $C$ need not be irreducible and is a finite chain of connected rational curves on $V$. If $V$ is smooth then these two notions coincide. See \cite{Kollar} for details.

Recall that rational and unirational varieties are rationally connected. In this section we will study rational chain connectedness of $W^1_X(2,L)$ when $g\geq 5$, and also determine the minimal length of  chains of rational curves required to connect any two points. 

 First we show the following.

\begin{theorem} Assume $X$ and $L$ are as in above theorem.
The Brill Noether subvariety $W^1_X(2, L)$ is rationally chain connected.
\end{theorem}
\begin{proof}
 A smooth hypersurface of degree $g$ in $\mathbb{P}^{3g-3}$ has its canonical bundle anti-ample. By \cite{MM}, it is rationally connected. It is covered by rational curves, say of degree $e$. The universal family of chains of rational curves of degree $e$ for the universal hypersurface of degree $g$ in $\mathbb{P}_L$ provides rational chain connectedness for $\mathcal{H}$.

Now consider the rational map
\begin{equation}\label{W}
\phi_L|_{\mathcal{H}}:\mathcal{H} \dashrightarrow W^1_X(2,L).
\end{equation}
Using Bertram's resolution as in \S 3,
$$
\Phi_L: \widetilde{\mathbb{P}}_L \rightarrow M_X(2,L)
$$
we consider the strict transform of $\mathcal{H}$ in this resolution and obtain the resulting morphism:
$$
\widetilde{\mathcal{H}} \rightarrow W^1_X(2,L).
$$
We further resolve the singularities of $\widetilde{\mathcal{H}}$ to obtain a birational morphism:
$$
{\widetilde{\mathcal{H}}}^{'} \rightarrow \mathcal{H}.
$$
Thus there is a morphism
$$
{\widetilde{\mathcal{H}}}^{'} \rightarrow W^1_X(2,L)
$$
which extends the rational map in $\eqref{W}$.

Now  $\mathcal{H}$ is rationally chain connected implies ${\widetilde{\mathcal{H}}}^{'}$ is rationally connected.
Note that for smooth varieties, rationally chain connected is equivalent to rationally connected.

This now implies that $W^1_X(2,L)$ is rationally chain connected.
\end{proof}

\begin{remark}
In fact we will improve the above result in the next subsection. We determine the minimal degree rational curves which cover $\mathcal{H}$ $($and hence also $W^{1}_{X}(2,L)$$)$  and the minimal length of the chain of lines on the hypersurface 
$\mathcal{H}$ required to connect any two points.
\end{remark}

\subsection{Minimal length of the chain of lines on $\mathcal{H}$}

Let $X\subset \mathbb{P}^n$ be an irreducible hypersurface of degree $d< n$.
Let $\mathcal{T}(X)$ denote the tangent sheaf of $X$ and $\mathbb{P} (\mathcal{T}(X))$ denote the 
projectivised tangent sheaf. Let $\mathcal{V}(X)\subset \mathbb{P} (\mathcal{T}(X))$ be the subscheme 
parametrising
pairs $(x,l)$ where $l$ is at least a $d$-fold tangent at $x$. Choose local 
coordinates $(x_1,...,x_n)$ around $x$ and the equation of $X$ can be written
as $f=f_1+f_2+...+f_d$ where $f_j$ are homogeneous terms of degree $j$.
The fiber of $\mathcal{V}(X)\longrightarrow X$ is given by $f_1=...=f_d=0$ 
\cite[p.283]{Kollar}. If $S\subset X$, then let $\mathcal{V}(S)$ denote the
inverse image of $S$ under the projection $\mathcal{V}(X)\longrightarrow X$.

Recall from \S 3.1, that $\mathbb{G}(1,n)$ is the Grassmanian parametrising one dimensional projective linear subspaces of $\mathbb{P}^{n}$. Suppose $S\subset X$ is an irreducible subvariety.
 Consider the morphism 
 $$
 \mathcal{V}(S)\longrightarrow \mathbb{G}(1,n), (s,l)\mapsto
[l].$$ 
Denote the pullback of the universal family by $\mathcal{L}(S)$. 
Then $\mathcal{L}(S)\subset \mathcal{V}(S)\times \mathbb{P}^n$. 
   
\begin{lemma}\label{4.33}
Let $X\subset \mathbb{P}^n$ be any hypersurface of degree $d$.
The family of lines $\mathcal{L}(S)$ surjects onto $X$ if $\mbox{dim}~S\geq d$.
\end{lemma}
\begin{proof} Let $p=(p_i)\in X$ be a general point.
Consider the polar hypersurfaces of $X$ with respect to $p$:
\begin{eqnarray*}
P_1(x,p) & : &\sum_{i=1}^n\frac{\partial{F}}{dx_i}(x).p_i=0 \\
P_2(x,p) & : & \sum\frac{\partial{F}}{\partial{x_i}\partial{x_j}}(x).p_ip_j=0 \\
&& \vdots \\
P_{d-1}(x,p)& : &\sum \frac{\partial{F}}{\partial{x_{i_1}}\partial{x_{i_2}}...\partial{x_{i_{(d-1)}}}}(x).p_{i_1}p_{i_2}...p_{i_{(d-1)}}=0.
\end{eqnarray*}

Then clearly the variety $\mathbb{P}(p)$ defined by the equations 
$P_1=P_2=...=P_{d-1}=0$ 
has a non-empty intersection $Z$ with $S$. This means that for $z\in Z$ the 
line 
$l_{zp}$ joining $p$ and $z$ has intersection $l_{zp}.X=d-1$ if $l_{zp}\not\subset X$. 
 In particular since dimension of any irreducible component of $Z$ is at least $1$ there is a $z\in Z$ and $z\neq p$
such that $p\in \mathcal{V}(z)$, i.e., $l_{zp}\subset X$. 
\end{proof}

Let $\dim$ $S\geq d$. Consider the fibre product $\mathcal{V}^{(2)}(S)=\mathcal{V}(S)\times_S\mathcal{V}(S)$ and the family
$\mathcal{L}^{(2)}(S)\longrightarrow \mathcal{V}^{(2)}(S)$ of connected 
chains of $2$ lines. Consider the cycle map $u:\mathcal{L}^{(2)}(S)\times_{\mathcal{V}^{(2)}(S)}\mathcal{L}^{(2)}(S)\longrightarrow X\times X$.
\begin{lemma}\label{4.44}
The cycle map $u$ surjects onto $ X \times X$ if $\mbox{dim}~S= 2d-1$.
\end{lemma}
\begin{proof} Let $p,q\in X$ be any two distinct points.
As in Lemma $(\ref{4.33})$, consider the polar hypersurfaces of $X$ with respect to $p$ and 
$q$:
$$P_1(x,p)=P_2(x,p)=...=P_{d-1}(x,p)=0,$$
$$P_1(x,q)=P_2(x,q)=...=P_{d-1}(x,q)=0.$$

The complete intersection defined by the polar hypersurfaces with respect to $p$ 
and $q$ intersects $S$ nontrivially since $\mbox{dim}~S=2d-1$. Call the
intersection $Z$. Now for $z\in Z$, $P_j(z,p)=0$ and 
$P_j(z,q)=0$ for $j\leq d-1$ imply that 
the points $p$ and $q$ lie on the complete intersection 
$\mathcal{V}_{d}(z)$. Since the dimension of any irreducible component of $Z$ is at least one, we can choose $z\neq p,z\neq q$.
In other words there are lines $l_{zp},l_{zq}\subset X$ and $(p,q)$ is in the
image of the cycle map $u$. 
\end{proof}

\begin{lemma}\label{chain}
Let $X\subset\mathbb{P}^n$ be any irreducible hypersurface of degree $d$. If $n\geq 2d$ then any 
two points on $X$ can be connected by a connected chain of two lines. 
\end{lemma}
\begin{proof} Suppose $X\subset \mathbb{P}^n$ is a hypersurface of
degree $d$ and satisfying $2d\leq n$. Then $\mbox{dim}~X\geq 2d-1$. Hence we can choose any subvariety 
$S\subset X$ of dimension $2d-1$.
By Lemma $(\ref{4.44})$, for any two points $p,q\in X$ there is a connected chain
of lines $l_1+l_2$ on $X$ such that $p\in l_1,q\in l_2$. 
\end{proof}

\begin{corollary}\label{chainH}
The hypersurface $\mathcal{H}$ is chain connected by lines of length two.
\end{corollary}
\begin{proof}
The degree of $\mathcal{H} \subset \mathbb{P}_L$ is $g$. Hence $\mathcal{H}$ fulfils the assumptions of Lemma $(\ref{chain})$, and we conclude the proof.
\end{proof}

Recall that a Hecke curve is a rational curve on  $M_X(2,L)$, and the universal family of Hecke curves determines a correspondence between moduli spaces of odd and even determinants. See \cite{Na-Ra}.

\begin{corollary}\label{Heckechain}
The Brill Noether loci $W^1_X(2,L)$ is chain connected by Hecke curves.
\end{corollary}
\begin{proof}
 The birational map
$$
\phi_L: {\mathbb{P}}_{L}\dashrightarrow M_X(2,L)
$$
takes lines to minimal degree rational curves on $M_X(2,L)$. These curves are the Hecke curves \cite{Hw}.
Hence by Corollary $(\ref{chainH})$, $W^1_X(2,L)$ is chain connected by Hecke curves. Over the open set  of $W^1_X(2,L)$ where $\Phi_L$ is birational, any two points are connected by Hecke curves of length two.
\end{proof}

\section{Chow groups of the hypersurface $\mathcal{H}$}

The Chow groups of hypersurfaces and complete intersections in projective spaces have attracted 
attention in the past few decades, leading to the Bloch Beilinson conjectures and effective bounds. Notable results are due to \cite{Ro}, \cite{N}, \cite{P}, \cite{elv}, \cite{o}, \cite{hi}.

We recall the following theorem, due to Otwinowska.

\begin{theorem}\label{4.2}
If $X\subseteq \mathbb{P}^{n}$ is a hypersurface of degree $d$ with $k(n-k+1)-\binom{d+k}{d}+1\geq 0$, then $CH_{k-1}(X)_{hom}\otimes \mathbb{Q}=0$.
\end{theorem}
\begin{proof}
See \cite[Corollary $(1)$]{o}.
\end{proof}

We apply above theorem to our situation as follows.

\begin{theorem}\label{l}
Let $X$ be a smooth projective curve of genus $g\geq 2$, $L$ be a line bundle on $X$ of degree $2g-1$ and $\mathcal{H}$ be the degree $g$ hypersurface in $\mathbb{P}_{L}$. Then $CH_{0}(\mathcal{H})_{\mathbb{Q}}\cong \mathbb{Q}$. Furthermore If $g$ is either $3$,$4$,$5$ or $6$, then $CH_{1}(\mathcal{H})_{\mathbb{Q}}\cong \mathbb{Q}$.
\end{theorem}
\begin{proof}
Let $(n,k,d)$ be $(3g-3,1,g)$. Then 
\[k(n-k+1)-\binom{d+k}{d}+1\]
\[=3g-3-(g+1)+1=2g-3\geq 1\]
as $g\geq 2$. Therefore on applying  Theorem $(\ref{4.2})$ to the hypersurface $\mathcal{H}$, we get that $CH_{0}(\mathcal{H})_{hom}\otimes_{\mathbb{Z}}\mathbb{Q}=0$. 
Choose $(n,d)=(3g-3,g)$. Then one checks easily that when $g\in\{3,4,5,6\}$, the only values of $k$ for which $k(n-k+1)-\binom{d+k}{d}+1\geq 0$, are $k=0,1,2$ and when $g\geq 7$, $k(n-k+1)-\binom{d+k}{d}+1< 0$ for all $k\geq 1$. Thus again on applying  Theorem $(\ref{4.2})$ to the hypersurface $\mathcal{H}$, we get that if $g\in\{3,4,5,6\}$, then $CH_{1}(\mathcal{H})_{hom}\otimes_{\mathbb{Z}}\mathbb{Q}=0$. 

Let $j_{*}:CH_{k}(\mathcal{H})_{\mathbb{Q}}\rightarrow CH_{k}(\mathbb{P}^{3g-3})_{\mathbb{Q}}$ denote the natural map induced by the inclusion $j:\mathcal{H}\rightarrow \mathbb{P}^{3g-3}$ for $k\geq 0$. But for $k\geq 0$, $j_{*}$ is injective if and only if $CH_{k}(\mathcal{H})_{hom}\otimes_{\mathbb{Z}}\mathbb{Q}=0$. As $j_{*}$ is surjective for all $k\geq 0$, the theorem follows.
\end{proof}
\begin{remark}
Theorem $(\ref{l})$ can also be deduced from \cite[Theorem $(1.5)$]{hi}.
\end{remark}

{\bf Acknowledments:}
The first author is supported by the DAE-Apex project ``Complex Algebraic Geometry" at the Institute of Mathematical Sciences and thanks the institute for its hospitality.

{\bf Data availability statement:}
Data sharing not applicable to this article as no datasets were generated or analysed during the current study.

{\bf Conflict of interest:}
There is no conflict of interest involved with this work.

\end{document}